\documentclass{file}
\usepackage{pat}
\usepackage{paralist}
\usepackage{dsfont}  
\usepackage[utf8x]{inputenc}
\usepackage{skull}
\usepackage{paralist}
\usepackage{graphicx}
\usepackage{wrapfig}
\usepackage{subfig}
\usepackage{thm-restate}
\usepackage{enumitem}
\usepackage[noend]{algorithmic}
\usepackage{bbm}  
\usepackage{listings}
\usepackage{blindtext}
\usepackage{titlesec}
\usepackage[normalem]{ulem}
\usepackage{cancel}
\DeclareRobustCommand{\bigO}{%
  \text{\usefont{OMS}{cmsy}{m}{n}O}%
}
\usepackage{todonotes}

\usepackage{tikz}
\usetikzlibrary{arrows.meta}
\usetikzlibrary{hobby}

\usepackage{etoolbox}

\usepackage{float}
\usepackage{tabularx}
\usepackage{array}
\usepackage{multirow}
\usepackage{makecell}

\usepackage[longnamesfirst,numbers,sort&compress]{natbib}

\usepackage[mathlines]{lineno}
\newcommand*\patchAmsMathEnvironmentForLineno[1]{%
 \expandafter\let\csname old#1\expandafter\endcsname\csname #1\endcsname
 \expandafter\let\csname oldend#1\expandafter\endcsname\csname end#1\endcsname
 \renewenvironment{#1}%
    {\linenomath\csname old#1\endcsname}%
    {\csname oldend#1\endcsname\endlinenomath}}%
\newcommand*\patchBothAmsMathEnvironmentsForLineno[1]{%
 \patchAmsMathEnvironmentForLineno{#1}%
 \patchAmsMathEnvironmentForLineno{#1*}}%
\AtBeginDocument{%
\patchBothAmsMathEnvironmentsForLineno{equation}%
\patchBothAmsMathEnvironmentsForLineno{align}%
\patchBothAmsMathEnvironmentsForLineno{flalign}%
\patchBothAmsMathEnvironmentsForLineno{alignat}%
\patchBothAmsMathEnvironmentsForLineno{gather}%
\patchBothAmsMathEnvironmentsForLineno{multline}%
}

\newcommand{\vol}[1]{\mathsf{Volt}(}

\DeclareMathOperator{\bn}{bn}
\DeclareMathOperator{\spn}{span}

\definecolor{brightmaroon}{rgb}{0.76, 0.13, 0.28}
\definecolor{linkblue}{rgb}{0, 0.337, 0.227}

\newrobustcmd{\onesub}{\mathord{\includegraphics{figs/one-sub}}}
\newrobustcmd{\leftup}{\mathord{\includegraphics{figs/left-up}}}

\title{\MakeUppercase{Basis number of bounded genus graphs}}
\author{Florian Lehner\thanks{Department of Mathematics, University of Auckland, 
New Zealand} \,\,\,\,\,Babak Miraftab\thanks{School of Computer Science, Carleton University, Ottawa, ON, Canada.} }

\date{}

\begin{document}

\maketitle

\begin{abstract}
The \textit{basis number} of a graph $G$ is the smallest integer $k$ such that $G$ admits a basis $B$ for its cycle space, where each edge of $G$ belongs to at most $k$ members of $B$.
In this note, we show that every non-planar graph that can be embedded on a surface with Euler characteristic $0$ has a basis number of exactly $3$, proving a conjecture of Schmeichel from 1981.
Additionally, we show that any graph embedded on a surface $\Sigma$ (whether orientable or non-orientable) of genus $g$ has a basis number of $\mathcal{O}(\log^2 g)$.

\end{abstract}

\section{Introduction}
Consider a graph $G$. The \defin{cycle space} of $G$, denoted as $\mathcal{C}(G)$, forms a vector space over the field $\mathbb{F}_2$. In this space, the elements are spanning subgraphs of $G$ in which all degrees are even, where the vector addition is defined by the symmetric difference of the edge sets of these subgraphs. 
Topologically, the cycle space is characterized by the first homology group $\mathsf{H}_1(G, \mathbb{Z}_2)$. 
The dimension of the cycle space, known as the \defin{Betti number} of $G$, is given by $\beta(G) = |E| - |V| + 1$ (see \cite{MR1280460}).
Cycle space theory has a wide range of applications, with numerous results documented in the literature \cite{polytime, shortest, minimumcycles, homotop, whitney}. 
Additionally, the cycle space has been extensively studied in the context of planar graphs.
\noindent A \defin{planar graph} is a graph that can be drawn in a plane without any of its edges crossing.
MacLane’s celebrated planarity criterion \cite{maclane1970combinatorial} characterizes planar graphs in terms of their cycle spaces. 

\begin{lem}[MacLane’s planarity criterion \cite{maclane1970combinatorial}] \label{maclane}A graph is planar if and only if there is a basis for its cycle space such that every edge lies in at most two elements of the basis.
\end{lem}
\noindent Schmeichel \cite{MR615307} introduced a specific term for the property described in MacLane’s planarity criterion.
\begin{defn}[{cf.\cite{MR615307}}]
The \defin{basis number} of a graph~$G$ denoted by $\bn(G)$ is defined as the smallest integer~$k$ for which~$G$ admits a basis~$\mathcal{B}$ that spans its cycle space, such that each edge in~$G$ belongs to no more than~$k$ members of~$\mathcal{B}$.
\end{defn}
\noindent For instance, by Maclane's theorem 
we know that a graph is planar if and only if the basis number of~$G$ is~$2$.
A \defin{$k$-basis} is a basis~$\mathcal B$ where every edge in~$G$ is contained in at most~$k$ members of~$\mathcal B$. 
There are several results on this topic, see Alsardary \cite{MR1844307}, Alzoubi and M. M. M. Jaradat \cite{MR2218262,MR2139817}, M. M. M. Jaradat \cite{MR1955410,MR2743750}, and McCall \cite{MR801601}.

One may inquire about the basis number of generalizations of planar graphs, as discussed in \cite{biedl}. In 1981, Schmeichel proved that if $G$ is a graph with genus $g$, then the basis number of $G$ is at most $2g + 2$ \cite[Theorem 6]{MR615307}. Additionally, he conjectured the following.

\begin{conj}{\rm\cite{MR615307}}
Every toroidal graph has a basis number $3$.
\end{conj}
\noindent In this paper, we prove the conjecture. More generally, we prove the following theorem.
\begin{restatable}{thm}{lowgenus}

\label{thm:lowgenus}
    Let $G$ be a non-planar graph which can be embedded on a surface of Euler characteristic $0$. Then $\bn (G) = 3$. 
\end{restatable}
\noindent We also investigate the relationship between genus and basis number for graphs of higher genus and show that the basis number is asymptotically much smaller than the genus of a graph.

\begin{restatable}{thm}{highgenus}

\label{thm:highgenus}
Let $G$ be a graph embedded on a (orientable or non-orientable) surface $\Sigma$ of genus $g$. 
Then $\bn(G)=\bigO(\log(g)^2)$.
\end{restatable}
\section{Preliminaries}

Throughout this short note, we shall use $\log$ for the base-$2$ logarithm and $\ln$ for the natural (base-$e$) logarithm.

Let us write $O_g$ for the \defin{orientable surface of genus $g$}, that is, a surface homeomorphic to a sphere with $g$ handles, and $N_g$ for the \defin{non-orientable surface of (non-orientable) genus $g$}, that is, a surface homeomorphic to a sphere with $g$ cross-caps. The well-known classification of compact surfaces states that every compact surface without boundary is homeomorphic to either $O_g$ for some $g \geq 0$, or $N_g$ for some $g \geq 1$.

An \defin{embedding} of a graph $G$ on a surface $\Sigma$ is a drawing of $G$ on $\Sigma$ where no edge crossings are allowed. Whenever we talk about embeddings in this paper, it will be assumed that $\Sigma$ is a compact surface without boundary, and that each region of $\Sigma \sm G$ is homeomorphic to an open disc known as a \defin{face}. 

Let $G=(V,E)$ be a graph embedded on a surface $\Sigma$ and let $F$ denote the set of faces of the embedding. The well-known Euler formula states that $|V| - |E| + |F| = 2 - 2g$ if $\Sigma$ is orientable of genus $g$, and $|V| - |E| + |F| = 2 - g$ if $\Sigma$ is non-orientable of genus $g$. The quantity $\chi = |V| - |E| + |F|$ is often called the Euler characteristic of the surface $\Sigma$; this only depends on $\Sigma$ and not on the graph, so we denote it by $\chi(\Sigma)$.

A \defin{(topological) cycle} on a surface $\Sigma$ is a homeomorphic image of the unit circle $\mathbb S^1$
in $\Sigma$. A \defin{(graph theoretical) cycle} in a graph $G$ is a connected, 2-regular subgraph of $G$. Note that if $G$ is embedded in $\Sigma$, then every graph theoretical cycle in $G$ corresponds to a topological cycle in $\Sigma$. Since these are the only cycles of interest in this paper, we may drop the distinction between the two notions; in other words, when talking about cycles we always mean cycles in an embedded graph which (by slight abuse of notation) are identified with their respective embeddings. A cycle $C$ is called \defin{separating} if $\Sigma \sm C$ has more than one connected component, and \defin{non-separating} otherwise.

The \defin{cycle space} of a graph $G = (V,E)$ is the set of all subgraphs of $G$ in which all degrees are even. This can be seen as a vector space over $\mathbb F_2$, more specifically a subspace of $\mathbb F_2^{|E|}$ by identifying each subgraph with the characteristic vector of the edges contained in it. A \defin{cycle basis} is a basis of the cycle space.

Note that elements of the cycle space are not necessarily cycles, but there are cycle bases consisting only of cycles. For instance, it is well known that if $T$ is a spanning tree of $T$, then the \defin{fundamental cycles} with respect to $T$, that is, the cycles containing precisely one edge not in $T$, form a cycle basis of $G$. In particular, the dimension of the cycle space of a graph $G=(V,E)$ is $|E|-|V|+1$. 

Assume that $G$ is embedded on a surface $\Sigma$. Note that each edge $e$ either lies in the closures of two different faces, or in the closure of exactly one face. To each face $F$ of the embedding we can define an element of the cycle space called the \defin{face boundary} of $F$, consisting of all edges in the closure of $F$ which also lie in the closure of some other face. By slight abuse of notation, we will sometimes also refer to a face $F$ as an element of the cycle space; this will always mean the face boundary of $F$. Each edge either appears in $0$ or in $2$ face boundaries, thus the sum of all face boundaries in the cycle space is $0$, or in other words, every face boundary is the sum of all other face boundaries. Moreover, every separating cycle $C$ can be written as the sum of the face boundaries of the faces contained in one connected component of $\Sigma \sm C$.

In general, the face boundaries alone do not generate the whole cycle space of an embedded graph. However, the following lemma can be deduced from the main result of~\cite{MR754916}.

\begin{lem}
    \label{lem:face-basis}
    Let $G$ be a graph embedded on a surface $\Sigma$, let $\mathcal F$ be the set of face boundaries of the embedding, and let $T$ be a spanning tree of $G$.
    There is a set $U$ of $2-\chi(\Sigma)$ edges not in $T$ such that the fundamental cycles of edges in $U$ with respect to $T$ together with any $|\mathcal F|-1$ elements of $\mathcal F$ form a cycle basis for $G$.
\end{lem}

We point out that the fundamental cycles in this lemma are necessarily non-separating since otherwise they would not be linearly independent from the face boundaries.

Recall that a \defin{$k$-basis} of the cycle space is a basis $\mathcal B$ such that every edge is contained in at most $k$ members of $\mathcal B$, and the \defin{basis number} of a graph $G$, denoted by $\bn (G)$ is the minimal $k$ for which there is a $k$-basis. Since every edge is contained in at most two face boundaries, \Cref{lem:face-basis} immediately implies that the basis number of a graph drawn on a surface $\Sigma$ is at most $4 - \chi(\Sigma)$. If $\Sigma$ is orientable of genus $g$, this means that $\bn(G) \leq 2+2g$, if $\Sigma$ is non-orientable of genus $g$ , then $\bn(G) \leq 2+g$.

\section{Basis number of graphs embedded on surfaces with low genus}

In this section show that the basis number of every non-planar graph which can be embedded on the torus $O_1$, the projective plane $N_1$, and the Klein bottle $N_2$ is $3$. 
For the projective plane this immediately follows from the discussion after \Cref{lem:face-basis}: the Euler characteristic is 1, therefore the facial cycles together with one additional cycle generate the cycle space.

The torus and the Klein bottle both have Euler characteristic 0, so we need to add two cycles to the set of face boundaries to generate the whole cycle space. This gives a generating set where each edge is contained in at most 4 of the elements: two facial cycles and at most two of the additional cycles. Recall that we need to remove precisely one face boundary from this generating set to obtain a basis. In particular, if there is a face boundary $f_0$ which contains all edges in the intersection of the two additional cycles, then removing this face boundary yields a $3$-basis. 

The following easy lemma allows us to replace a pair of cycles with another pair of cycles whose intersection is contained in a face boundary. We point out that when talking about unions or intersections of elements of the cycle space, we always mean unions or intersections of the corresponding edge sets. 

\begin{lem}
    \label{lem:replacement}
    Let $G$ be a graph with cycle space $\mathcal C(G)$, let $\mathcal F \subseteq \mathcal C(G)$ and let $x,y \in \mathcal C(G)$. Assume that we can find $h,k \in \spn (\mathcal F)$ and $f_0\in \mathcal F$ such that
    \begin{align*}
        x \cap y & \subseteq h \cup k \cup f_0,&
        x \cap k & \subseteq h \cup y \cup f_0,&
        h \cap y & \subseteq x \cup k \cup f_0,&
        h \cap k & \subseteq x \cup y \cup f_0.
    \end{align*}
    Then $\spn(\mathcal F \cup \{x,y\}) = \spn(\mathcal F \cup \{x+h,y+k\})$ and $(x+h) \cap (y+k) \subseteq f_0$.
\end{lem}

\begin{proof}
    Since $h$ and $k$ are contained in $\spn (\mathcal F)$, replacing $x$ by $x+h$ and $y$ by $y+k$ does not change the span. For the second part, note that if an edge $e$ is in both $x+h$ and $y+k$, then it is contained in precisely one of the sets $x$ and $h$, and in precisely one of the sets $y$ and $k$. Consequently it must be contained in one of the sets $(x\cap y) \sm(h \cup k)$, $(x\cap k) \sm(h \cup y)$, $(h\cap y) \sm(x \cup k)$, and $(h\cap k) \sm(x \cup y)$. By assumption, all of these sets are contained in $f_0$.
\end{proof}

\lowgenus*


\begin{proof}
    Let $\Sigma$ be the surface the graph is embedded in. Denote by $\mathcal F$ the set of face boundaries. By \Cref{lem:face-basis}, there are two fundamental cycles $x$ and $y$ with respect to some spanning tree $T$ of $G$ such that $\mathcal F \cup \{x,y\}$ generates the cycle space.

    Our goal is to find $h,k \in \spn (\mathcal F)$ and $f_0 \in \mathcal F$ satisfying the conditions of \Cref{lem:replacement}. Note that in this case $\mathcal F \cup \{x+h,y+k\}$ generates the cycle space. Since $\mathcal F$ is linearly dependent, $\mathcal B := (\mathcal F \sm \{f_0\}) \cup \{x+h,y+k\}$ still generates the cycle space. If an edge is contained in both $x+h$ and $y+k$, it is only contained in one element of $\mathcal F \sm f_0$ because $(x+h) \cap (y+k) \subseteq f_0$. Thus every edge is contained in at most $3$ elements of $\mathcal B$, therefore proving the theorem.

    It remains to find $h$, $k$, and $f_0$. If $x \cap y = \emptyset$, then we may set $h = k = \emptyset$, and $f_0 \in \mathcal F$ arbitrary; this satisfies the conditions of \Cref{lem:replacement} because the left hand side in each of the inclusions in the condition is $\emptyset$. 
    
    Hence we may assume that $x \cap y \neq \emptyset$. Since $x$ and $y$ are both fundamental cycles with respect to the tree $T$, the set $x \cap y$ is the edge set of a path $P_{x,y}$ in $T$ and the sets $x \sm y$ and $y \sm x$ are the edge sets of (vertex) disjoint paths $P_x$ and $P_y$ connecting the endpoints of $P_{x,y}$. In particular, the subgraph $H$ of $G$ spanned by the edge set $x \cup y$ consists of two vertices and three disjoint paths between these two vertices.

    We claim that the embedding of $H$ has exactly one face. Indeed, $H$ contains precisely $3$ cycles: $x$, $y$, and $x+y$. If $H$ had more than one face, then one of these cycles would have to be separating, and thus it would be contained in $\spn(\mathcal F)$. But then $(\mathcal F\sm\{f\}) \cup \{x,y\}$ would be linearly dependent for every $f \in \mathcal F$ and therefore could not be a basis.

    Cutting along $x \cup y$ we therefore obtain a surface homeomorphic to a closed disk (a fundamental polygon of $\Sigma$ which we denote by $\Sigma'$). The embedding of $G$ lifts to an embedding of a graph $G'$ in $\Sigma'$ where the paths $P_{x,y}$, $P_x$, and $P_y$ in $G$ each correspond to two paths in $G'$ embedded on the boundary of $\Sigma'$. Consequently, the boundary of $\Sigma'$ can be decomposed into 6 paths (the sides of the fundamental polygon) corresponding to the two sides of $P_{x,y}$, $P_x$, and $P_y$, which we denote by $p_{x,y}$, $p_{x,y}'$, $p_x$, $p_x'$, $p_y$, and $p_y'$, respectively. 
    
    Note that the sides corresponding to the two sides of the same path in $G$ cannot be adjacent on the boundary. Further, note that we can swap the roles of the different paths by replacing $x$ or $y$ with $x+y$. Thus, up to swapping the roles of the paths we end up in one of the two scenarios depicted in \Cref{fig:polygons}. We note that if the surface is the torus, then we end up with the polygon on the left side of \Cref{fig:polygons}, and if the surface is the Klein bottle, then we end up with the polygon on the right side (but this won't matter for the rest of the proof).

\begin{figure}
\centering
\subfloat[\centering ]{{\includegraphics[scale=0.90]{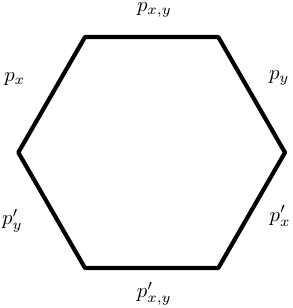}}}
\qquad
\subfloat[\centering ]{{\includegraphics[scale=0.90]{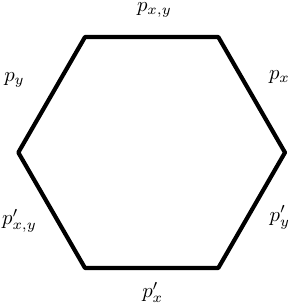} }}%
\caption{Fundamental polygons obtained by cutting along the edges of $x \cup y$. The left side results from cutting the torus, the right side from cutting the Klein bottle.}
\label{fig:polygons}
\end{figure}
The faces of the embedding of $G'$ are in one-to-one correspondence with the faces of the embedding of $G$, and by slight abuse of notation, we identify the faces of the two embeddings. Sums of face boundaries will always be taken in $G$, not $G'$.

Before we proceed with the proof, we need one more piece of notation.
Let $Q$ be a simple curve in $\Sigma'$ such that both endpoints lie in the boundary of $\Sigma'$. Embed the fundamental polygon in $\mathbb R^2$, and let $Q'$ be a curve in $\mathbb R^2$ connecting the endpoints of $Q$ and intersecting the fundamental polygon only in these endpoints. Then $Q \cup Q'$ is a simple closed curve. 
We say that a subset of $\Sigma'$ lies below $Q$ if it is contained in the bounded region of $\mathbb R^2 \sm (Q \cup Q')$, and above $Q$ if it is contained in the unbounded region of $\mathbb R^2 \sm (Q \cup Q')$; we note the roles of `below' and `above' may switch depending on the choice of $Q'$.
We say that $Q$ separates two subsets of $\Sigma'$ if one of them lies below $Q$ and one of them lies above $Q$.

We now distinguish two cases. Either some face boundary (in the embedding of $G'$) contains a path $Q$ connecting vertices in two non-consecutive sides of $\Sigma'$, or no face boundary contains such a path.

In the first case, note that by swapping the roles of the boundary parts we may assume that neither of the two endpoints of $Q$ is an interior vertex of $p_{x,y}$ or $p_{x,y}'$. Therefore we may without loss of generality assume that all edges of $p_{x,y}$ which do not lie on $Q$ lie below $Q$, and all edges of $p_{x,y}'$ which do not lie on $Q$ lie above $Q$. Let $\mathcal I$ be the set of faces whose interior lies below $Q$, and let $h = k = \sum_{f \in \mathcal I} f$, see \cref{fig_1}.
\begin{figure}
    \centering
\includegraphics[scale=0.9]{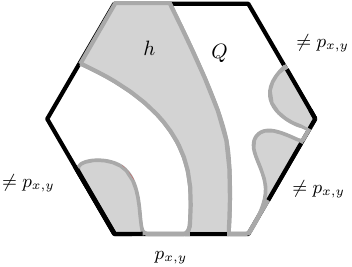}
    \caption{A path $Q$ separating the fundamental polygon into an area below $Q$ and an area above $Q$}
   \label{fig_1}
\end{figure}

Since $k = h$ we immediately obtain $x \cap k \subseteq h$ and $k \cap y \subseteq h$. Next, let $e \in x \cap y$ be an edge not contained in $Q$. Then one of the two copies of $e$ lies below $Q$ and the other one does not. Thus $\mathcal I$ contains exactly one of the two faces whose boundary $e$ lies in, and therefore $e \in h$. Finally, consider and edge $e \in h \sm (x \cup y)$. Since $e\in h$, we know that $\mathcal I$ contains exactly one of the two faces whose boundary $e$ lies in, and since $e \notin x \cup y$ this implies that $e$ is contained in $Q$.

Now assume that we are in the second case, so there is no face boundary that contains vertices in two non-consecutive sides of the fundamental polygon. Without loss of generality assume that $p_{x,y}$ shares a vertex with $p_x$ and a vertex with $p_y$. 

Let $v_0$ be the last vertex of $p_{x,y}$ (when traversing the path from $p_x$ to $p_y$) such that there is a face $f_0$ incident to $v_0$ and some vertex of $p_x$. Let $\mathcal I_x$ be the set of faces incident to edges of $p_{x,y}$ before $v_0$. Let $\mathcal I_y$ be the set of faces incident to edges of $p_{x,y}$ after $v_0$. Note that by tracing the boundary of $f_0$, we obtain a curve $Q$ separating all edges of $p_{x,y}$ before $v_0$ from all edges of $p_{x,y}$ after $v_0$. Thus $Q$ also separates the faces in $\mathcal I_x$ from the faces in $\mathcal I_y$, and face boundaries of faces in $\mathcal I_x$ and $\mathcal I_y$ can only intersect in $Q$.

Let $h = \sum_{f\in \mathcal I_x} f$ and let $k = \sum_{f\in \mathcal I_y} f$. Since we are in the second case, none of the faces in $\mathcal I_x \cup \mathcal I_y$ contains a vertex of $p_{x,y}'$, $p_x'$, or $p_y'$. Each edge of $p_{x,y}$ is contained in precisely one face boundary in $\mathcal I_x \cup \mathcal I_y$; this implies that $x \cap y \subseteq h \cup k$. Next note that $Q$ separates all faces in $\mathcal I_x$ from $p_y$. Since $Q$ is contained in the boundary of $f_0$ and intersects $p_x$, it cannot intersect $p_y$. This implies that no element of $\mathcal I_x$ contains an edge in $p_y$, and thus $h \cap y \subseteq x \cap y \subseteq x$. By definition of $v_0$, no element of $\mathcal I_y$ contains an edge in $p_x$, so $x \cap k \subseteq x \cap y \subseteq y$. Finally, as noted above face boundaries of faces in $\mathcal I_x$ and $\mathcal I_y$ can only intersect in $Q$, and since $Q$ is contained in the boundary of $f_0$ we conclude that $h \cap k \subseteq f_0$.
\end{proof}

\section{Basis numbers of graphs on surfaces of higher genus}

\begin{lem}
    \label{lem:subgraph}
    Let $G$ be a graph embedded on a surface $\Sigma$. Then $G$ contains a subgraph $H$ with Betti number $\beta(H) = 2-\chi(\Sigma)$ such that the cycles of $H$ together with the facial cycles of $G$ generate $\mathcal C(G)$. Consequently, $\bn(G) \leq \bn(H) + 2$.
\end{lem}

\begin{proof}
    This is a straightforward consequence of \Cref{lem:face-basis}: the graph $H$ consists of a spanning tree $T$ of $G$ together with the $2-\chi(\Sigma)$ edges whose fundamental cycles together with the facial cycles of $G$ generate the cycle space of $G$.
\end{proof}

In order to use the above lemma recursively, we bound the (orientable) genus of the graph $H$. It is straightforward to see that every graph with a Betti number of $1$ is planar. Additionally, Kuratowski demonstrated that this is also true for all graphs with Betti numbers $2$ and $3$. Later, Milgram and Ungar generalized Kuratowski's result.

\begin{lem}{\rm\cite{MR0469801}}\label{lem:max}
Let $ g(\beta) $ denote the largest (orientable) genus occurring among graphs with Betti number $ \beta $. 
Then the following holds:
\[
\frac{\beta}{2} + \frac{1}{2} \left( 3(\beta - 1) / (\log(\beta - 1)) \right) \leq g(\beta) \leq \frac{\beta}{2} - \frac{\beta}{4 \log \beta}.
\]
\end{lem}

\highgenus*

\begin{proof}
Let $H$ the graph obtained by \cref{lem:subgraph}. The Betti number of $H$ is at most $2g$ (in fact it is at most $g$ if $\Sigma$ is non-orientable). 
Next it follows from \cref{lem:max} that the orientable genus of $H$ is at most $g - \frac{g}{(2 \log (2g))}$.
Let $ \bn(g) $ denote the largest basis number occurring among all graphs with orientable genus $g$. 
Then by the above discussion
\[
\bn(g)\leq 2+\bn \left(g- \frac{g}{2 \log (2g)}\right).
\]

We use this recursion to show by induction that there exists a positive real number $ M $ and a positive integer $ g_0 $ such that $\bn(g) \leq M \log(g)^2$
for all $ g \geq g_0 $.

To this end, first let 
\[
     f(g) =  2 \log(g) \log \left(1 - \frac{1}{2 \log (2g)}\right) + \log \left(1 - \frac{1}{2 \log (2g)}\right)^2,
\]
and note that
\[
\lim_{g \to \infty } f(g) = -\frac{1}{\ln 2}
\]
which can be readily verified using de l'Hospital's rule. Pick $0 < \epsilon < \frac{1}{\ln 2}$, and let $g_0$ be large enough such that $f(g) < -\epsilon$ for every $g \geq g_0$. Then choose $M \in \mathbb R$ large enough such that $2+2g < M \log (g)^2$ for every $g < g_0$, and such that $\epsilon M \geq 2$.

The discussion after \cref{lem:face-basis} shows that $\bn(g) \leq 2+2g$. Our choice of $M$ implies that $\bn (g) \leq M \log (g)^2$ for every $g < g_0$. Now let $g \geq g_0$, and assume that $\bn(\ell)\leq M \log (\ell)^2$ for every $\ell<g$. Then
\begin{align*}
    \bn(g)& \leq  2+\bn(g- \frac{g}{2 \log (2g)})\\
    & \leq  2+ M\left(\log \left(g- \frac{g}{2 \log (2g)}\right)\right)^2\\
    & \leq  2+ M\left(\log (g)+\log \left(1- \frac{1}{2 \log (2g)}\right)\right)^2\\
    & \leq  2+ M\left(\log (g)^2+2\log(g)\log \left(1- \frac{1}{2 \log (2g)}\right)+\log \left(1- \frac{1}{2 \log (2g)}\right)^2\right)\\
    & \leq  2+ M\left(\log (g)^2+ f(g)\right)\\
    &\leq \log (g)^2 + 2 - \epsilon M\\
    &\leq \log (g)^2,
\end{align*}
thus finishing the induction step.
\end{proof}

\noindent {\bf{Acknowledgement}}\\
{The second author would like to thank Saman Bazargani for the helpful discussions during the preparation of this paper.
}
\bibliographystyle{plainurlnat}
\bibliography{bngg.bib}
\nocite{*}

\end{document}